\documentclass[a4paper]{amsart}

\usepackage{amssymb}
\usepackage{amsthm}

\usepackage{url}

\usepackage{hyperref}
 \hypersetup{colorlinks, breaklinks,
            	linkcolor = blue,
				urlcolor = blue,
				anchorcolor = blue,
				citecolor = blue}
						
\numberwithin{equation}{section}

\theoremstyle{plain} \newtheorem{theorem}{Theorem}[section]
\theoremstyle{plain} \newtheorem{lemma}[theorem]{Lemma}
\theoremstyle{plain} \newtheorem{claim}[theorem]{Claim}
\theoremstyle{plain} \newtheorem{cor}[theorem]{Corollary}
\theoremstyle{plain} \newtheorem{prop}[theorem]{Proposition}
\theoremstyle{plain} 
\theoremstyle{plain} \newtheorem{fact}[theorem]{Fact}
\theoremstyle{definition} \newtheorem{remark}[theorem]{Remark}
\theoremstyle{definition}
\theoremstyle{definition}

\DeclareMathOperator{\codim}{codim}
\DeclareMathOperator{\Span}{span}
\DeclareMathOperator{\dist}{dist}
\DeclareMathOperator{\diam}{diam}
\DeclareMathOperator{\proj}{proj}
\DeclareMathOperator{\conv}{conv}

\newcommand{\rr}{\mathbb{R}}

\newcommand{\R}{\mathbb{R}}
\newcommand{\Z}{\mathbb{Z}}

\newcommand{\Rn}{\mathbb{R}^n}
\newcommand{\N}{\mathbb{N}}

\newcommand{\eps}{\varepsilon}
\newcommand{\de}{\delta}

\newcommand{\su}{\subset}
\newcommand{\sm}{\setminus}
\newcommand{\cK}{{\mathcal{K}}}
\newcommand{\cG}{{\mathcal{G}}}

\newcommand{\J}{J}
\newcommand{\dists}{d}
\newcommand{\cA}{\mathcal{A}}

\begin{document}

\title[Small unions of affine subspaces and skeletons via Baire category]{Small unions of affine subspaces and skeletons via Baire category}

\author{Alan Chang, Marianna Cs\"ornyei, Korn\'elia H\'era and Tam\'as Keleti}

\date{}



\address
{Department of Mathematics, The University of Chicago,
5734 S. University Avenue, Chicago, IL 60637, USA}

\email{ac@math.uchicago.edu}

\email{csornyei@math.uchicago.edu}

\address
{Institute of Mathematics, E\"otv\"os Lor\'and University, 
P\'az\-m\'any P\'e\-ter s\'et\'any 1/c, H-1117 Budapest, Hungary}

\email{herakornelia@gmail.com}

\email{tamas.keleti@gmail.com}

\thanks{The third and fourth authors were supported  
by the Hungarian National Research, Development and Innovation Office –- NKFIH 104178, 
and the third author was also supported by the \'UNKP-16-3 New National Excellence Program of the Ministry of Human Capacities.} 

\begin{abstract}
Our aim is to find the minimal Hausdorff dimension of the union of scaled and/or rotated copies of the $k$-skeleton of a fixed polytope centered at the points of a given set. For many of these problems, we show that a typical arrangement in the sense of Baire category gives minimal Hausdorff dimension. 
In particular, this proves a conjecture of R. Thornton.

Our results also show that Nikodym sets are typical among all sets which contain, for every $x\in\R^n$,  
a punctured hyperplane  $H\setminus \{x\}$ through $x$.
With similar methods we also construct a Borel subset of $\Rn$ 
of Lebesgue measure zero
containing a hyperplane at every positive distance from every point.
\end{abstract}

\maketitle

\section{Introduction}

E. Stein \cite{St} proved in 1976 that for any $n\ge 3$, 
if a set $A\su\Rn$ contains a sphere centered at each point of
a set $C\su\Rn$ of positive Lebesgue measure, then $A$ 
also has positive Lebesgue measure. 
It was shown by Mitsis \cite{Mi} that the same holds if
we only assume that 
$C$ is a Borel subset of $\R^n$ of Hausdorff dimension greater than 1. 
The analogous results are also true in the case $n=2$; this was proved 
independently by Bourgain \cite{Bo} and Marstrand \cite{Mar} for 
circles centered at the points of an 
arbitrary set 
$C\su\R^2$
of positive Lebesgue measure, 
and by Wolff \cite{Wo} for  $C\subset\R^2$ 
of Hausdorff dimension greater than $1$.
In fact, Bourgain proved a stronger result, which extends 
to other curves with non-zero curvature.

Inspired by these results, the authors in \cite{KNS} studied what happens if the circles are replaced by axis-parallel squares. They 
constructed a closed set $A$ of Hausdorff dimension $1$ that contains the boundary of an axis-parallel square centered at each point in $\R^2$
(see \cite[Theorem 1.1]{KNS}).
Thornton studied in \cite{Th} the higher dimensional versions: the problem when $0\le k < n$ and $A\su\Rn$ contains
the $k$-skeleton of an $n$-dimensional axis-parallel 
cube centered at every point of a compact set of given 
dimension $d$ for some fixed $d\in[0,n]$. (Recall that the \emph{$k$-skeleton} of a polytope is the union of its $k$-dimensional faces.) He found the smallest possible dimension of 
such a compact $A$ in the cases when 
we consider box dimension and packing dimension. He conjectured that
the smallest possible Hausdorff dimension of $A$ is $\max(d-1,k)$, 
which would be the generalization of \cite[Theorem 1.4]{KNS}, 
which addresses the case $n = 2, k = 0$.

In this paper we prove Thornton's conjecture
not only for cubes but for general polytopes of $\Rn$. It turns out  that it plays an important role whether 0 is contained in one of the $k$-dimensional affine subspaces defined by the $k$-skeleton of the polytope (see Theorem~\ref{skeleton}). This is even more true if instead of just scaling, we also allow rotations. 
In this case,
we ask what the minimal Hausdorff dimension of a set is that contains a scaled and rotated copy of the $k$-skeleton of a given polytope centered at each point of $C$. Obviously, it must have dimension at least $k$ if $C$ is nonempty. 
It turns out that this is sharp: we show that there is a Borel set of dimension $k$ that contains a scaled and rotated copy of the $k$-skeleton of a polytope centered at each point of $\R^n$, \emph{provided that 0 is not in any of the  $k$-dimensional affine subspaces defined by the $k$-skeleton}. On the other hand, if 0 belongs to one of these affine subspaces, then the problem becomes much harder (see Remark~\ref{Kakeya}).

\medskip
As mentioned above at the end of the second paragraph, a (very) special case of Theorem~\ref{skeleton}, namely, when $n=2$ and $S$ consists of the 4 vertices of a square centered at the origin, was already proved in \cite{KNS}. 
Our proof of Theorem~\ref{skeleton} is much simpler than
the proof in \cite{KNS}. In fact, in all our results mentioned above, we will show that, in the sense of Baire category, the minimal dimension is attained by residually many sets. As it often happens, it is much easier to show that some properties hold for residually many sets than to try to construct a set for which they hold. 
In our case, after proving residuality for $k$-dimensional affine subspaces, we automatically obtain residuality for countable unions of $k$-dimensional subsets of $k$-dimensional affine subspaces, hence $k$-skeletons.
If we allow rotations but do not allow scaling, the question becomes: what is the minimal Hausdorff
dimension of a set that contains a rotated copy of the $k$-skeleton of a given polytope centered at each point of $C$? We do not know the answer to this question for a general compact set $C$. However, as the following simple example shows, it is no longer true that a typical construction has minimal dimension.

Let $C \su \rr^2$ denote the unit circle centered at $0$, and let the ``polytope'' be a single point of $C$. Then $\{0\}$ is a set of dimension 0 that contains, centered at each point of $C$, a rotated copy of our ``polytope''. (That is, it contains a point at distance 1 from each point of $C$.) On the other hand, it is easy to show that, if $A$  contains a \emph{nonzero} point at distance 1 from each point of $C$, then $A$ has dimension at least 1. In particular, a ``typical'' $A$ has dimension 1 and not 0. The same example also shows that the minimal dimension can be different depending on whether the ``polytope'' consists of one point or two points.

However, we will show that a typical construction does have minimal dimension, provided that $C$ has full dimension, i.e., $\dim C=n$ for $C\subset\R^n$. In this case, the minimal (as well as typical) dimension of a set $A$ that contains a rotated copy of the $k$-skeleton of a polytope centered at each point of $C$ is $k+1$.
Somewhat surprisingly, we obtain that the smallest possible dimension (and also the typical dimension) is still $k+1$ if we want the
$k$-skeleton of a rotated copy of the polytope of \emph{every size} centered at every point. 

\medskip

Let us state our results more precisely.
Throughout this paper, by a \emph{scaled copy} of a fixed set $S\su\Rn$ we mean a 
set of the form $x+rS=\{x+rs\ :\ s\in S \}$, where $x\in\Rn$ and $r>0$. We say that $x+rS$ is a scaled copy of $S$ \emph{centered at $x$}.  
(That is, the center of $S$ is assumed to be the origin.) Similarly, a 
\emph{rotated copy} of $S$ centered at $x \in \R^n$ is $x+T(S) =\{x+T(s)\ :\ s\in S \}$, where $T\in SO(n)$. Combining these two, we define a
\emph{scaled and rotated copy} of $S$ centered at $ x\in\Rn$ by $x+rT(S)=\{x+rT(s)\ :\ s\in S \}$, where $r>0$ and $T\in SO(n)$. 

In this paper we will consider only Hausdorff dimension, and we will denote by $\dim E$ the Hausdorff dimension of a set $E$. We list here the special cases of our results when the polytope is a cube and the set of centers is $\Rn$. (The first statement was already proved in \cite{Th}.)

\begin{cor}\label{c:dim}
For any integers $0\le k<n$, the minimal dimension of a Borel set $A\su\R^n$ 
that contains the $k$-skeleton of
\begin{enumerate} 
\item a scaled copy of a cube centered at every point of $\Rn$ is $n-1$;
\item a scaled and rotated copy of a cube centered at every point of $\Rn$ is $k$;
\item a rotated copy of a cube centered at every point of $\Rn$ is $k+1$;
\item a rotated cube of every size centered at every point of $\Rn$ is $k+1$.
\end{enumerate}
In fact, the same results hold if the $k$-skeleton of a cube is replaced 
by any $S\su\R^n$ with $\dim S=k$ that can be covered by a countable union
of $k$-dimensional affine subspaces that do not contain $0$.
\end{cor}

\medskip
For $k = n-1$ it is natural to ask if, in addition to dimension $k + 1 = n$, we can also guarantee positive Lebesgue measure in the settings (3) and (4). As we will see, we cannot guarantee positive measure. We show that there are residually many \emph{Nikodym sets}, i.e., sets of measure zero which contain a punctured hyperplane through every point. %
The existence of Nikodym sets in $\R^n$ for every $n\ge 2$ was proved by Falconer \cite{Fa86}.
We also obtain residually many sets of measure zero which contain a hyperplane at every positive distance from every point.
By combining our these two results, we get the following.
\begin{cor}\label{c:zeromeasure}
Let $S\su\R^n$ $(n\ge 2)$ be a set that can be covered by countably many
hyperplanes and suppose that $0\not\in S$. Then there exists a set of
Lebesgue measure zero that contains a scaled and rotated copy of $S$
of every scale centered at every point of $\R^n$.
\end{cor}
Note that here we need only the assumption $0\not\in S$ 
(which clearly cannot be dropped), while in Corollary~\ref{c:dim} we needed
the stronger assumption that the covering affine subspaces do not contain $0$. Also, Corollary~\ref{c:zeromeasure} is clearly false for $n = 1$.

\medskip
One can ask what happens for those sets $S$ to which neither the classical results nor our results can be applied. %
One of the simplest such case is when, say, $n=1$ and $S=C-1/2$, where $C$ is the classical triadic Cantor set in the interval $[0,1]$. We do not know how large a set $A$ can be that contains a scaled
copy of $S$ centered at each $x\in\R$. Does it always have positive Lebesgue measure, or Hausdorff dimension at least $1$?
In \cite{LP} \L{}aba and Pramanik construct random Cantor sets for which such a set must have positive Lebesgue measure, and by the result of M\'ath\'e \cite{Mt}, there exist Cantor sets for which such a set $A$ can have zero measure. Hochman \cite{Ho} and Bourgain \cite{Bo2} prove that for any porous Cantor set $C$ with $\dim C > 0$, such a set $A$ must have Hausdorff dimension strictly larger than $\dim C$ and at least $1/2$.

\medskip

Finally we remark that T. W. K\"orner \cite{Ko} observed in 2003 that small Kakeya-type sets can be constructed using Baire category argument.
He proved that if we consider the Hausdorff metric on the space of all 
compact sets that contain line segments
in every possible direction between two fixed parallel line segments, 
then in this space, residually many sets have zero Lebesgue measure. 
As we will see, in our results we obtain
residually many sets in a different type of metric space: 
we consider Hausdorff metric in a ``code space''.

\section{Scaled copies}%

In this section we consider only scaled (not rotated) copies of $S$. We will prove the following theorem:
 
\begin{theorem}\label{skeleton} 
Let $S$ be the $k$-skeleton of an arbitrary polytope in $\R^n$ 
for some $0\le k<n$, and let $d\in[0,n]$ be arbitrary.
\begin{itemize}
\item[(i)] Suppose that 
$0$ is not contained in any of the $k$-dimensional affine subspaces 
defined by $S$.
Then the smallest possible dimension of a compact set $A$ that contains 
a scaled copy of $S$ 
centered at each point of some $d$-dimensional compact set $C$ is $\max(d-1,k)$.
\item[(ii)] Suppose that 
$0$ is contained in at least one of the $k$-dimensional affine subspaces 
defined by $S$.
Then the smallest possible dimension 
of a compact set $A$ that contains 
a scaled copy of $S$ 
centered at each point of some $d$-dimensional compact set $C$ is $\max(d,k)$.
\end{itemize}
\end{theorem}

Thornton's conjecture mentioned in the introduction is clearly a special case of part (i) of this theorem.

\medskip
In fact, our main goal is to study a slightly different problem, from 
which we can deduce the results above.
Our aim is to find for a given ``skeleton'' $S$ 
and for a given 
nonempty compact set of centers $C$ (instead of a given $S$ and a given $\dim C$) the smallest possible value of $\dim A$, where $A$
contains a scaled copy of $S$ centered at each point of $C$.

We will study the case when $S$ is the $k$-skeleton of a polytope, or more generally, the case when $S$ is a countable union $S=\bigcup S_i$, where each $S_i$ is contained in an affine subspace $V_i$.
We will assume that $C$ is compact and nonempty. Our aim is to show that, in the sense of Baire category, a typical set $A$ that contains a scaled copy of $S$ centered at each point of $C$ has minimal dimension. 

Let us make this more precise. 
Fix a nonempty compact set $C\su\Rn$ and a non-degenerate closed interval $I\subset (0,\infty)$.
In what follows, we view $C \times I$ as a parametrization of the space of certain scaled copies of a given set $S \subset \R^n$; in particular, $(x, r) \in C \times I$ corresponds to the copy centered at $x$ and scaled by $r$.
Let $\cK$ denote the space of all compact sets $K\subset C\times I$ that have full projection onto $C$. (That is, for each $x\in C$ there is an $r\in I$ with $(x,r)\in K$.) We equip $\cK$ with the Hausdorff metric. Clearly, $\cK$ is a closed subset of the space of all compact subsets of $C\times I$, and hence it is a complete metric space. In particular, the Baire category theorem holds for $\cK$, so we can speak about a typical $K\in\cK$ in the Baire category sense: a property $P$ holds for a typical $K\in\cK$ 
if $\{K\in\cK : P \textrm{ holds for } K\}$ is residual in $\cK$,
or equivalently,
if there exists a dense $G_\delta$ 
set $\cG\su\cK$ such that the property holds for every $K\in\cG$.

Let $A$ be an arbitrary set that contains a scaled copy of $S\su\Rn$ 
centered at each point of $C$.
First we show an easy lower estimate on $\dim A$, which in some important cases will 
turn out to be sharp.
Let $C'$ denote the orthogonal projection of $C$ onto $W:=\Span\{S\}^\perp$.
(As usual, we denote by $\Span\{S\}$ the linear span of $S$, 
so it always contains the origin.)
For every point $x'\in C'$ there exists an $x\in C$
such that the projection of $x$ onto $W$ is $x'$,
and there exists an $r>0$
such that $x+rS\su A$ and hence
$x+rS\su (x'+\Span\{S\}) \cap A$.
Since for any $x'\in C'\su W = \Span\{S\}^\perp$ 
the set $(x'+\Span\{S\}) \cap A$ contains a 
scaled copy of $S$, we obtain by the general Fubini type inequality
(see e.g. in \cite{Fa85} or \cite{Fa90})
\begin{equation}
\label{triviineq}
\dim A \ge \dim C' + \dim S.
\end{equation}

Now let $K\in\cK$ and $S\su\Rn$ and consider
\begin{equation}\label{defAKS}
A=A_{K,S}:=\bigcup_{(x,r)\in K}x+rS. 
\end{equation}
Note that $A_{K,S}$ contains a scaled copy of $S$ centered at each point of $C$,
so by the previous paragraph,
\begin{equation}
\label{triviAKS}
\dim A_{K,S} \ge \dim C' + \dim S.
\end{equation}

The following lemma shows that for a typical $K\in\cK$ we have equality in
\eqref{triviAKS} if $S$ is an affine subspace.

\begin{lemma}\label{main}
Let $V$ be an affine subspace of $\Rn$, let $\emptyset\neq C\su\R^n$ be compact,
and let $C'$ denote the projection of $C$ onto $\Span\{V\}^\perp$.
Then for a typical $K\in\cK$, and for $A_{K,V}$ defined by \eqref{defAKS},
$$\dim A_{K,V}=\dim C'+\dim V.$$
\end{lemma}

We postpone the proof of this lemma and first study some of its corollaries. 
Suppose that $S$ is a countable union $S=\bigcup S_i$, where each $S_i$ is a
subset of an affine subspace $V_i$. %
Let $C_i'$ denote the orthogonal projection of $C$ onto $W_i:=\Span\{V_i\}^\perp$. Since a countable intersection of 
residual sets is residual, and since the Hausdorff dimension of a countable union of sets is the supremum of the Hausdorff dimension of the individual sets, it follows that for a typical $K\in\cK$, 
$$\dim A_{K,S}=
\dim\left(\bigcup_i A_{K,S_i}\right)\le \sup_i (\dim C_i'+\dim V_i).$$
On the other hand, if $A$ contains a scaled copy of $S=\bigcup_i S_i$ centered at each $x\in C$, then applying \eqref{triviineq} to each $S_i$, we get
$\dim A \ge \dim C'_i+\dim S_i$ for each $i$ and thus 
$\dim A\ge \sup_i (\dim C_i'+\dim S_i)$.
Therefore, we obtain the following theorem:

\begin{theorem}\label{thm} Let $C$ be an arbitrary nonempty compact subset 
in $\R^n$, and let $S=\bigcup_{i=1}^\infty S_i$, where each $S_i$ is a 
subset of an affine subspace $V_i$. %
Let $C_i'$ denote the orthogonal projection of $C$ onto $\Span\{V_i\}^\perp$. Then:
\begin{itemize}
\item[(i)] For every set $A$ that contains a scaled copy of $S$ centered at each point of $C$,
$$\dim A\ge \sup_i (\dim C_i'+\dim S_i).$$
\item[(ii)] For a typical $K\in\cK$, the set 
$A=A_{K,S}$ defined by \eqref{defAKS} 
contains a scaled copy of $S$ centered at each point of $C$ and
$$\dim A \le \sup_i (\dim C_i'+\dim V_i).$$
Furthermore, if $S$ is compact then so is $A$.
\end{itemize}
\end{theorem}

Let $W_i=\Span\{V_i\}^\perp$. Note that if $0\not\in V_i$ then $\dim W_i=n-\dim V_i-1$. 
Therefore if $\dim C=n$, $k<n$, $\dim S=k$, and for every $i$ we
have $0\not\in V_i$ and $\dim V_i=k$, then 
$\sup_i \dim S_i=k$ and $\dim C'_i=n-k-1$ for every $i$, so 
Theorem~\ref{thm} gives $\dim A=n-1$, which proves
the general version of (1) of Corollary~\ref{c:dim}.

\medskip
So far we studied the problem of finding the minimal Hausdorff dimension of a set $A$ that contains a copy of a given set $S$ centered at each point of a given set $C$.
Now we turn to the problem when, instead of $S$ and $C$, we are only given $S$ and $d=\dim C$. 
We suppose that $\dim S_i = \dim V_i$ for each $i$, so the lower and 
upper estimates in (i) and (ii) agree.

Since clearly 
$\dim C_i' \ge \max(0, \dim C - \codim W_i)$, 
where $\codim W_i$ denotes the co-dimension of the linear space $W_i$, therefore
Theorem~\ref{thm}(i) gives
$$
\dim A \ge \sup_i (\max(0, d - \codim W_i) + \dim S_i).
$$
In order to show that this estimate is sharp when $\dim S_i=\dim V_i$, by
Theorem~\ref{thm}(ii), it is enough to find a compact set $C\su\Rn$ for which 
$\dim C_i' = \max(0, \dim C - \codim W_i)$ holds for each $i$.
This can be done by the following claim, which we will prove later.
\begin{claim}\label{falconer}
For each $i\in\N$, let $W_i$ be a linear subspace of $\R^n$ of co-dimension 
$l_i\in\{0,1,\ldots,n\}$.
Then for every $d\in[0,n]$ there exists a $d$-dimensional compact set $C\su\Rn$ whose projection onto $W_i$ has dimension $\max(0,d-l_i)$ for each $i$.
\end{claim}

Therefore Theorem~\ref{thm} and Claim~\ref{falconer} give the following.

\begin{cor}\label{corgen} 
Suppose that $S=\bigcup_{i=1}^\infty S_i$, where each $S_i$ is a 
subset of an affine subspace $V_i$ with $\dim S_i=\dim V_i$. 
For each $i$, let $W_i=\Span\{V_i\}^\perp$.
Let $d\in [0,n]$ be arbitrary. 
Then the smallest possible dimension of a set $A$ that contains a scaled copy of $S$ centered at each point of some $d$-dimensional set $C$ is 
$\sup_i (\max(0, d - \codim W_i) + \dim S_i)$.
\end{cor}

Now we claim that 
Theorem~\ref{skeleton} is a special case of Corollary~\ref{corgen}.
Indeed, if $S$ is a $k$-skeleton of a polytope, 
then for each $i$ we have $\dim S_i = \dim V_i =k$, and 
$W_i$ has co-dimension either $k+1$ if 
$0\not\in V_i$, or $k$ if $0\in V_i$. 
Thus $\max(0, d - \codim W_i) + \dim S_i$ is either $\max(k,d-1)$ if $0\not\in V_i$, or $\max(k,d)$ if $0\in V_i$.

It remains to prove Claim~\ref{falconer} and Lemma~\ref{main}.
The following simple proof is based on an argument that 
was communicated to us by
K.~J.~Falconer.

\begin{proof}[Proof of Claim~\ref{falconer}]
We can clearly suppose that $d>0$ and $l_i\in\{1,\ldots,n-1\}$.
For $0<s\le n$, Falconer \cite{Fa94} 
introduced $\cG^s_n$ as the class of those $G_\de$ 
subsets $F\subset\R^n$ for which $\bigcap_{i=1}^\infty f_i(F)$
has Hausdorff dimension at least $s$
for all sequences of similarity transformations
$\{f_i\}_{i=1}^\infty$.
Among other results, Falconer proved
that $\cG^s_n$ is closed under countable intersection, and
if $F_1\in\cG^s_n$ and $F_2\in\cG^t_m$ then $F_1\times F_2 \in \cG^{s+t}_{n+m}$.
Examples of sets of $\cG^s_n$ with Hausdorff dimension exactly $s$ are also
shown in \cite{Fa94} for every $0<s\le n$.

For $l<d$, let $E_l\in\cG^{d-l}_{n-l}$ with $\dim E_l=d-l$, and
for $l\ge d$ let $E_l$ be a dense $G_\delta$ subset of $\R^{n-l}$ with 
$\dim E_l=0$.
Let $F_l= E_l \times \R^l\subset\R^{n-l}\times\R^l$. 
Clearly, the projection of $F_l$ onto $\R^{n-l}$ has Hausdorff dimension 
$\max(0,d-l)$.

Now we show that $F_l\in\cG^d_n$. This follows from the product rule
mentioned above if $l< d$. In the case $l\ge d$,
we need to prove that $\dim(\bigcap_{i=1}^\infty f_i(E_l\times \R^l))\ge d$
for any sequence of similarity transformations $\{f_i\}_{i=1}^\infty$.
Let $V$ be an $(n-l)$-dimensional subspace of $\Rn$ which is generic in
the sense that it intersects all the countably many $l$-dimensional affine
subspaces $f_i(\{0\}\times \R^l)$ in a single point. 
Then for each translate $V+x$ of $V$, the set
$f_i(E_l\times\R^l)\cap (V+x)$ is similar to the dense $G_\de$ set $E_l$,
hence $(\bigcap_{i=1}^\infty f_i(E_l\times \R^l))\cap (V+x)$ is
nonempty for each $x$, which implies that indeed
$\dim(\bigcap_{i=1}^\infty f_i(E_l\times \R^l))\ge l \ge d$.

For each $i$, let $H_i$ be a rotated copy of $F_{l_i}$ with projection
of Hausdorff dimension $\max(0,d-l_i)$ onto $W_i$.
Since each $H_ i$ is of class  $\cG^{d}$, the intersection  
$D:= \bigcap_{i=1}^\infty H_i$ is of class $\cG^{d}$. In particular, its Hausdorff dimension is at least $d$. It is also clear that the projection of $D$ onto each $W_i$ has Hausdorff
dimension at most $\max(0,d-l_i)$.

Now $D$ has all the required properties except that 
it might have Hausdorff dimension larger than $d$, and it is not compact
but $G_\de$. 
If $\dim D>d$, then let $C$ be a compact subset of $D$ with Hausdorff
dimension $d$. 
Then for each $i$, the projection of $C$ onto $W_i$ is at most $\max(0,d-l_i)$,
but it cannot be smaller since $W_i$ has co-dimension $l_i$.
If $\dim D=d$ then let $D_j$ be compact subsets of $D$ with $\dim D_j\to d$
and let $C$ be a disjoint union of shrunken converging copies of $D_j$ and 
their limit point. 
\end{proof}

\begin{proof}[Proof of Lemma~\ref{main}]
By \eqref{triviAKS}, it is enough to show that
$\dim A_{K,V}\le \dim C'+\dim V$ holds for a typical $K\in\cK$.
Write $V=v+V_0$ where $V_0$ is a $k$-dimensional linear subspace, $v\in\R^n$
and $v\perp V_0$. 
Without loss of generality we can assume that $v=0$ or $|v|=1$.
Let $x'$ denote the projection of a point $x$ onto $\Span\{V\}^\perp$,
and let $\proj x\in\R$ denote the projection of $x$ onto $\R v$.
(Clearly, if $v=0$, then $\proj x=0$.)

Let $\cK^n$ denote the space of all nonempty compact subsets of $\R^n$, equipped with the Hausdorff metric. Then 
$$A=A_{K,V}=\bigcup_{(x,r)\in K}x'+(\proj x+r)v+V_0,$$ so
$$\dim A=\dim V_0+\dim\left(\bigcup_{(x,r)\in K}x'+(\proj x+r)v\right)=k+\dim F(K),$$
where $F:\,\cK\to\cK^n$ is defined by $$F(K)=\bigcup_{(x,r)\in K}x'+(\proj x+r)v.$$ It is easy to see that $F$ is continuous. 

Since for every open set $G\su\Rn$ and for every compact set $K\su G$ we have
$\dist(\R^n\sm G, K)>0$, it follows that for any open set $G\subset\R^n$, $\{K\in\cK^n: K\subset G\}$ is an open subset of $\cK^n$. 
Consequently, for any $s, \de, \eps>0$, the set of those compact sets $K\in\cK^n$
that have an open cover $\bigcup G_i$ where 
$\sum_i (\diam G_i)^s<\eps$ and $\diam G_i<\de$ for each $i$ is an open
subset of $\cK^n$. Therefore 
for any $s>0$, $\{K\in \cK^n: \dim K \le s\}$ is a $G_\de$ subset of $\cK^n$.
Since $F$ is continuous, $\{K\in \cK : \dim F(K) \le s\}$ is a $G_\de$ subset of $\cK$.

\medskip
We finish the proof by showing that $\{K\in \cK : \dim F(K) \le \dim C'\}$ is dense. To obtain this, for every compact set $L\in\cK$ we construct another compact set $K\in \cK$ arbitrary close to $L$, such that 
$\{\proj x+r:(x,r)\in K\}$ is finite and so $F(K)$ is covered by a finite union of copies of $C'$. 
For a given $L\in\cK$, such a $K\in\cK$ can be constructed by choosing a sufficiently small $\eps>0$
and letting
$$K:=\{(x,r):\,\exists r'\text{ s.t. } (x,r')\in L,\,\proj x+r\in\eps\Z,\,|r-r'|\le\eps\}.\mbox{\qedhere}$$
\end{proof}

\section{Scaled and rotated copies}%

In this section, we study the problem when we are allowed to \emph{scale and rotate} copies of $S$. That is, now our aim is to
find for a given set $S\su\R^n$ and a nonempty compact set of centers $C\su\Rn$ 
the minimal possible value of $\dim A$, where $A$ contains a scaled and
rotated copy of $S$ centered at each point of $C$.
(That is, for every $x\in C$, there exist $r>0$ and $T\in SO(n)$ such that
$x+rT(S)\su A$.)

For a fixed nonempty compact set $C\su\Rn$ and a closed interval $I\subset (0,\infty)$,
let $\cK'$ denote the space of all compact sets $K\su C\times I \times SO(n)$
that have full projection onto $C$. 
We fix a metric on $SO(n)$ that induces the natural topology 
and equip $\cK'$ with the Hausdorff metric.
Then $\cK'$ is also a complete metric space, so again we can talk about
typical $K\in\cK'$ in the Baire category sense.
Now for $K\in\cK'$ and $S\su\Rn$, we let
\begin{equation}
\label{A'}
A'_{K,S}: = \bigcup_{(x,r,T)\in K} x+rT(S).
\end{equation}
Note that $A'_{K,S}$ contains a scaled and rotated copy of $S$
centered at each point of $C$.

Again, first we consider the case when $S$ is an affine subspace, but
we now exclude the case when $S$ contains $0$.

\begin{lemma}
\label{l:rotated}
Let $V$ be an
affine subspace of $\Rn$ such that $0\not\in V$ and
let $C\su\Rn$ be an arbitrary nonempty compact set. 
Then for a typical $K\in\cK'$, and for $A'_{K,V}$ defined by \eqref{A'},
$$
\dim A'_{K,V}=\dim V.
$$
\end{lemma}

\begin{proof}
Clearly it is enough to show that $\dim A'_{K,V}\le\dim V$
holds for a typical $K\in\cK'$.
For any $N\in\N$, we define $F_N':\cK'\to\cK^n$ by
$F_N'(K)=A'_{K,V}\cap [-N,N]^n$. It is easy to see that $F'_N$ is continuous.
Then exactly the same argument as in the proof of  Lemma~\ref{main}
gives that $\{K\in\cK' : \dim F_N'(K)\le s\}$ is a $G_\de$ 
subset of $\cK'$, which implies that 
$\{K\in\cK' : \dim A'_{K,V}\le s\}$ is also $G_\de$.

So it remains to prove that  $\{K\in\cK' : \dim A'_{K,V}\le \dim V\}$ is dense. 
Fix $\eps>0$. 
Then, by compactness and since $0\not\in V$, 
there exists an $N=N(\eps)\in\N$ and $(\dim V)$-dimensional affine
subspaces $V_1,\ldots,V_{N}$ such that
for any $(x,r,T)\in C\times I \times SO(n)$
there exists $(r',T')\in I \times SO(n)$ 
within $\eps$ distance of $(r,T)$ such that $x+r'T'(V)=V_i$ for some $i\le N$.
Thus, given any compact set $L\in\cK'$ and $\eps>0$,
we can take
$$
K=\{(x,r',T')\ :\ 
x+r' T'(V)\in\{V_1,\ldots,V_{N}\}\ \} \cap L_\eps,
$$
where 
$$
L_\eps=\{(x,r',T')\ :\ \exists(r, T)\text{ s.t. } (x,r,T)\in L,\ \dist((r',T'),(r,T))\le\eps\}.
$$ 
It follows that $K \in \cK'$ and the (Hausdorff) distance between $K$ and $L$ is at most $\eps$. Furthermore, $\dim A'_{K,V}=\dim V$, since
$A'_{K,V}$ can be covered by
finitely many $(\dim V)$-dimensional affine spaces.
\end{proof}

By taking a countable intersection of residual sets we obtain the following corollary of Lemma~\ref{l:rotated}, 
which clearly implies the general form of (2) of Corollary~\ref{c:dim}.

\begin{theorem}
Let $C$ be an arbitrary nonempty compact subset in $\R^n$, $k<n$ and let $S\su\R^n$ be a $k$-Hausdorff-dimensional set that can be covered by a countable union of $k$-dimensional affine subspaces that do not contain $0$.
Then for a typical $K\in\cK'$, the set $A'_{K,S}$ 
contains a scaled and rotated copy of $S$ centered at every point of $C$, and
$\dim A'_{K,S}=\dim S$. %
\end{theorem}

\begin{remark}
\label{Kakeya}
If $0\in V$ and $V$ is $k$-dimensional then 
a scaled and rotated copy of $V$
centered at $x$ is a $k$-dimensional affine subspace that contains $x$.
Therefore a set $A$ that contains a scaled and rotated copy
of $V$ centered at every point 
of $C$ is a set that contains a $k$-dimensional affine subspace
through every point of $C$.
The Lebesgue measure of such an $A$ is clearly bounded below by the Lebesgue measure of $C$. By generalizing the planar result of Davies \cite{Da} to higher dimensions, Falconer \cite{Fa86} proved there is such an $A$ which attains this lower bound. In Section~\ref{s:Alan} we show that the Lebesgue measure of a typical such $A$ is in fact this minimum.
On the other hand, 
to find the minimal dimension of such an $A$ is closely related to
the Kakeya problem, especially in the special case $k=1$, and for some
nontrivial $C$ this problem is as hard as the Kakeya problem.
\end{remark}

\section{Rotated copies: dimension}
\label{s:rotated}

Now we study what happens if we allow rotation 
but do not allow scaling.
As we mentioned in the introduction, it is \emph{not} true that for a general nonempty compact set of centers $C$, a typical construction has minimal dimension. However, we will show that this is true provided that $C$ has full dimension. 

The following lower estimate can be found in \cite{HKM}:

\begin{fact}\label{t:atleastkplus1}
Let $0 \leq k < n$ be integers, and let $S \su \R^n$ be a $k$-Hausdorff-dimensional set that can be covered by a countable union
of $k$-dimensional affine subspaces that do not contain $0$. 
Let $\emptyset\neq C \su \R^n$ and $A \su \R^n$ be such that 
for every $x \in C$, there exists a rotated copy of $S$ centered at $x$ contained in $A$. 
Then $\dim A \geq \max\{ k, k+ \dim C - (n-1) \}$. 

In particular, if $\dim C=n$ then  $\dim A \ge k+1$.
\end{fact}

\begin{remark}
If instead of fixing $C$, we fix only the dimension $d$ of $C$, and $S$ can be covered by one $k$-dimensional affine subspace $V$, then 
the following simple examples show that the estimate in Fact~\ref{t:atleastkplus1} is sharp. 
Without loss of generality we can assume that $V$ is at unit distance from $0$.
For $d \leq n-1$, we can take $A = \R^k \times \{0\} \subset \R^{n}$ and take $C$ to be a $d$-dimensional subset of $\R^k \times S^{n-k-1}$, 
where $S^m$ denotes the unit sphere in $\R^{m+1}$ centered at $0$.
For $d = n-1 + s$, where $s \in [0, 1]$, 
let $E \subset \R^{n-k}$ be an $s$-dimensional subset of a line
and let $F \subset \R^{n-k}$ be the set with a copy of $S^{n-k-1}$ 
centered at every point of $E$. It is easy to show that $\dim F = n-k-1+s$. 
Let $C = \R^k \times F$ and $A = \R^k \times E$.
In both cases $A$ contains a rotated copy of $S$ centered at every point of $C$,
$\dim C=d$ and $\dim A=\max\{k, k+\dim C - (n-1)\}$.

If $S$ can be covered by two distinct $k$-dimensional affine subspaces but
cannot be covered by one, then this question becomes much more difficult. 
Consider, for example, the case 
when $S$ consists of two points, both at distant $1$ from $0$, so now
$A$ contains two distinct points at distance $1$ from every point of a $1$-dimensional set $C \subset \R^2$. The discussion in the introduction implies that if we take $C = S^1$, then $\dim A \geq 1$. We do not know if there exists a set $C$ with $\dim C = 1$ for which there is such a set $A$ with $\dim A < 1$.
\end{remark}

Our goal is to show that for every fixed $C$ with $\dim C = n$, the estimate $\dim A \geq k + 1$ in Fact~\ref{t:atleastkplus1} is always sharp. Moreover, we construct sets of Hausdorff dimension $k+1$ that contain
the $k$-skeleton of an
$n$-dimensional rotated polytope of \emph{every} size centered at \emph{every} point.
More precisely, we want to construct a set $A$ that contains
a rotated copy of every positive size of a given set $S\su\R^n$ centered at 
every point of a given nonempty compact set $C$. 
(That is, for every $x\in C$ and $r>0$ there exists $T\in SO(n)$ such that
$x+rT(S)\su A$.) Instead of every $x \in C$ and $r>0$ we will guarantee only every $(x,r)$
from each fixed nonempty compact set $\J\su \R^n \times (0,\infty)$. By taking 
countable unions, we get the desired construction for every $(x, r) \in \R^n \times (0, \infty)$.

For a fixed nonempty compact set $\J\su \R^n \times (0,\infty)$, let $\cK''$ denote the space of all compact sets 
$K\su \J \times SO(n)$ that have full projection onto $\J$.
Again, by taking a metric on $SO(n)$ that induces the natural topology
and equipping $\cK''$ with the Hausdorff metric, 
$\cK''$ is also a complete metric space, so again we can talk about
typical $K\in\cK''$ in the Baire category sense.

Now for any $K\in\cK''$ and $S\su\Rn$, the set
\begin{equation}
\label{A''}
A''_{K,S}: = \bigcup_{(x,r,T)\in K} x+rT(S)
\end{equation}
contains a rotated copy of $S$
of scale $r$ centered at $x$ for every $(x,r) \in \J$.
Note that taking $\J=C \times \{1\}$ gives us the special case when only 
rotation is used.

Again, we start with the case when  $S$ is a $k$-dimensional 
 ($0\le k<n$) affine subspace of $\Rn$
that does not contain the origin. 
Note that if $d=\dist(S,0)$ then $x+rT(S)$
is at distance $rd$ from $x$.
This motivates the following easy 
deterministic $(k+1)$-dimensional construction.


\begin{prop}
For any integers $0\le k<n$ there exists a Borel set $B\su\R^n$ of Hausdorff 
dimension $k+1$ that contains a $k$-dimensional affine subspace at every
positive distance from every point of $\R^n$.
\end{prop}

\begin{proof}
Let $W_1, W_2,\ldots$ be a countable collection of $(k+1)$-dimensional 
affine subspaces of $\R^n$ such that $B:=\bigcup_i W_i$ is dense. 
Then $B$ is clearly a Borel set $B\su\R^n$ of Hausdorff dimension $k+1$, 
so all we need to show is that for any fixed $x\in\R^n$ and 
$r>0$ the set $B$ contains a $k$-dimensional affine subspace at distance $r$
from $x$. Choose $i$ such that $W_i$ intersects the interior of the ball 
$B(x,r)$. Then the intersection of $W_i$ and the sphere $S(x,r)$ is a sphere
in the $(k+1)$-dimensional affine space $W_i$, and any $k$-dimensional 
affine subspace of $W_i\su B$ that is tangent to this sphere 
is at distance $r$ from $x$. 
\end{proof}

The proof of the following lemma is based on the same 
idea as in the construction above.

\begin{lemma}
\label{l:rotatedeverysize}
Let $0\le k<n$ be integers and $V$ be a $k$-dimensional
affine subspace of $\Rn$ such that $0\not\in V$.
Let $\J\su \R^n\times(0,\infty)$ be an arbitrary nonempty compact set.
Then for a typical $K\in\cK''$, and for $A''_{K,V}$ defined by \eqref{A''},
$$
\dim A''_{K,V}\le k +1.
$$
\end{lemma}

\begin{proof}
Without loss of generality we can assume that $V$ is at distance $1$ from the origin.

Let $A(n,k+1)$ be the space of all $(k+1)$-dimensional affine
subspaces of $\Rn$, equipped with a natural metric (for example the metric
defined in \cite[3.16]{Ma}), and 
let $W_1,W_2,\ldots$ be a countable dense set
in $A(n,k+1)$. Let $B=\bigcup_i W_i$. 

Exactly the same argument as in the proof of  Lemma~\ref{l:rotated}
gives that 
$\{K\in\cK'' : \dim A''_{K,V}\le s\}$ is $G_\de$ for any $s$, so
again it remains to prove that  $\{K\in\cK'' : \dim A''_{K,V}\le k +1\}$ is
dense in $\cK''$. 
Since $\dim B=k+1$, it is enough to show that 
$\{K\in\cK'' : A''_{K,V}\su B\}$ is dense in $\cK''$.

First we show that for any $(x,r,T)\in \J \times SO(n)$ 
and $\eps>0$, there exist $i\in\N$ and $T'\in SO(n)$ such that
$\dist(T,T')<\eps$ and $x+rT'(V)\su W_i$. 
We will also see from the proof that for the given $\eps>0$ and the above chosen $i$,
there exists a neighborhood of $(x,r,T)$ such that
for any $(x^*,r^*,T^*)$ from that neighborhood, 
there exists ${T^*}'\in SO(n)$ such that
$\dist(T^*,{T^*}')<\eps$ and $x^*+r^*{T^*}'(V)\su W_i$.
Hence, by the compactness of $\J \times SO(n)$,
for a given $\eps>0$, there exists an $N$ such that  we can choose an
$i\le N$ for every $(x,r,T)\in \J \times SO(n)$.

So fix $(x,r,T)\in \J \times SO(n)$ and $\eps>0$. 
Let $W$ be a $(k+1)$-dimensional affine subspace of $\R^n$ that contains $V$
such that $0<\dist(W,0)<\dist(V,0)=1$. We denote by $v$ be the  point
of $V$ closest to the origin, and let $V_0=x+rT(V)$, $v_0=x+rT(v)$ and $W_0=x+rT(W)$.
Then $S_0:=W_0\cap S(x,r)$ is a sphere
in $W_0$, and $V_0$ is the tangent of $S_0$ at the point $v_0$.
If $W_i$ is sufficiently close to $W_0$, then we can pick a point $v'_0\in S'_0:=W_i\cap S(x,r)$ close to
$v_0$, and a $k$-dimensional affine subspace $V'_0\su W_i$ close to 
$V_0$ that is the tangent of $S'_0$ at $v'_0$.
Then $V'_0$ is at distance $r$ from $x$ and it is as close to $V_0=x+rT(V)$ 
as we wish, so $V_0'=x+rT'(V)$ for some $T'\in SO(n)$
and $T'$ can be chosen arbitrarily close to $T$, 
which completes the proof of the claim of the previous paragraph.

Thus, for a given $L\in\cK''$ and $\eps>0$, if we let
$$
K=\{(x,r,T')\ : \exists i\le N,\exists\, T\text{ s.t. } 
   (x,r,T)\in L,\ \dist(T,T')\le \eps,\
   x+rT'(V)\su W_i\},
$$
then $K \in \cK''$ and the Hausdorff distance between $K$ and $L$ is at most $\eps$.
Furthermore,
$A''_{K,V}\su \bigcup_{i=1}^{N}W_i \su B$,
which completes the proof.
\end{proof}

The same statements hold if, instead of $S=V$, we consider any subset $S\subset V$. By taking a countable intersection of residual sets we obtain the following.

\begin{theorem}\label{t:generalkplus1}
Let $0\le k<n$ be integers and 
let $S=\bigcup_{i=1}^\infty S_i$, where each $S_i$ is a 
subset of a $k$-dimensional affine subspace $V_i$ with 
$0\not\in V_i$. 
Let $\J\su \R^n \times (0,\infty)$ be an arbitrary nonempty compact set.
Recall that $\cK''$ denotes the space of all compact sets 
$K\su \J \times SO(n)$ that have full projection onto $\J$.

Then
for a typical $K\in\cK''$, the set $A''_{K,S}$ defined by \eqref{A''} 
is a closed set with $\dim A''_{K,S}\le k+1$, and
for every $(x,r)\in\J$, there exists a $T\in SO(n)$ such that 
$x+rT(S)\su A''_{K,S}$.
\end{theorem}

We can see from Fact \ref{t:atleastkplus1} that the estimate $k+1$ above is sharp, provided that $\dim S=k$
and $J\supset C\times\{r\}$ for some $r>0$ and $C\su\R^n$ with $\dim C=n$.
This gives the general version of (3) and (4) of Corollary~\ref{c:dim}.  

\begin{remark}
In Theorem~\ref{t:generalkplus1} we obtain a rotated and scaled copy of $S$
for every $(x, r) \in \J$ inside a set of Hausdorff 
dimension $k+1$. We claim that using a similar
argument as in \cite[Remark 1.6]{JJKM} we can also move $S$ \emph{continuously} 
inside a set of Hausdorff dimension $k+1$ so that during this motion we get
$S$ in every required position. 
Indeed, let $K$ be a fixed (typical) element
of $\cK''$ guranteed by Theorem~\ref{t:generalkplus1}
such that $\dim A''_{K,S}\le k+1$.
Since $K$ is a nonempty compact subset of the
metric space $\conv(\J)\times SO(n)$, 
where $\conv$ denotes the convex hull,
there exists a continuous function
$g:C_{1/3}\to \conv(\J)\times SO(n)$ 
on the classical Cantor set $C_{1/3}$
such that $g(C_{1/3})=K$.
All we need to do is to extend this map 
continuously to $[0,1]$ such that 
$\dim A''_{g([0,1]),S} \le k+1$.
For each complementary interval $(a,b)$ of the Cantor set, we define $g$ on $(a,b)$ in such a way that $g$ is smooth on $[a,b]$ and that the diameter of $g([a,b])$ is at most a constant multiple of the distance between $g(a)$ and $g(b)$. This gives the desired extension since the union of
the sets of the form $x+rT(S)$ ($(x,r,T)\in g((a,b))$) will be a countable union of smooth
$k+1$-dimensional manifolds, so 
$\dim A''_{g((a,b)),S} = k+1$.

Note that if $\J=C \times \{1\}$ then we get only congruent copies.
So in particular, for any $k<n$, the $k$-skeleton
of a unit cube can be continuously moved by rigid motions in $\Rn$ 
within a set of Hausdorff dimension $k+1$ in such
a way that the center of the cube goes through every point of $C$, 
or by joining such motions, through every point of $\Rn$.
\end{remark}

\section{Rotated copies: measure}
\label{s:Alan}

In this section, we study what happens when we place a rotated punctured hyperplane through every point. We show that typical arrangements of this kind have Lebesgue measure zero and are hence Nikodym sets. Using similar methods, we also show that typical arrangements of  placing a rotated hyperplane at every positive distance from every point have measure zero. We use $| \cdot |$ to denote the Lebesgue measure.

Let $e_1 = (1,0,\ldots,0) \in \R^n$ and $H = \{(y_1, \ldots, y_n) \in \R^n : y_1 = 0\}$. By a \emph{rotated hyperplane at distance $r \in [0, \infty)$ from $x \in \R^n$}, we mean a set of the form $x + rT(e_1) + T(H)$ for some $T \in SO(n)$. Note that we now allow $r$ to be $0$, and that $x + rT(e_1) + T(H)$ differs from $x + rT(e_1+H)$ when $r = 0$.

Fix a nonempty compact set $\J \subset \R^n \times [0, \infty)$. As in Section \ref{s:rotated}, we let $\cK''$ denote the space of compact sets $K \subset \J \times SO(n)$ that have full projection onto $\J$.

In this section we prove the following result.

\begin{theorem}
\label{theorem:typical-davies-nikodym}
For a typical $K \in \cK''$, the set 
\[
\bigcup_{(x, r, T)\in K} (x + rT(e_1) + T(H)) \setminus \{x\}
\]
has measure zero.
\end{theorem}

Note that if $r = 0$, then $(x + rT(e_1) + T(H)) \setminus \{x\}$ is $x + T(H\setminus\{0\})$, so we are placing a rotated copy of the \emph{punctured hyperplane} $H\setminus\{0\}$ through $x$. Thus, if we consider the case $J = C \times \{0\}$ for some compact set $C \subset \R^n$, we see that typical arrangements give rise to Nikodym sets. We also obtain our claim in Remark~\ref{Kakeya} that if we place an un-punctured hyperplane through every point in $C$, the typical arrangement of this kind has Lebesgue measure equal to $|C|$. 

By taking countable unions of sets of the form in Theorem~\ref{theorem:typical-davies-nikodym}, we obtain the following:

\begin{cor}
There is a set of measure zero in $\R^n$ which contains a hyperplane at every positive distance from every point as well as a punctured hyperplane through every point.
\end{cor}

\subsection{Translating cones}
\label{subsec:cones}

In this section we introduce the main geometric construction for proving Theorem~\ref{theorem:typical-davies-nikodym}. 
This construction is done in $\R^2$ and we will later see how to apply it to the $n$-dimensional problem.
Our geometric arguments are similar to those used to construct Kakeya needle sets of arbitrarily small measure, see e.g. \cite{Be}. 

For $-\frac{\pi}{2} < \phi_1 < \phi_2 < \frac{\pi}{2}$, we define
\begin{align*}
D(\phi_1, \phi_2) 
&= 
\{ (r \sin\theta, r \cos\theta) : r \in \R, \theta \in [\phi_1, \phi_2] \}.
\end{align*}
In other words, $D(\phi_1, \phi_2) \subset \R^2$ denotes the double cone bounded by the lines through the origin of signed angles $\phi_1, \phi_2$ with respect to the $y$-axis. (Note in particular that our sign convention measures the angles in the clockwise direction.)

Our geometric construction begins by partitioning $D = D(\phi_1, \phi_2)$ into finitely many double cones $\{ D_i\}$. Next, we translate each $D_i$ downwards to a new vertex $v_i \in D_i \cap \{y_2 < 0\}$ to obtain $\widetilde D_i := v_i +  D_i$. Our goal is to choose the $\{D_i\}$ and the $\{v_i\}$ so that the resulting double cones $\{ \widetilde D_i\}$ satisfy the following three properties. 

First, the $\{\widetilde D_i\}$ should have considerable overlap (and hence small measure) in a strip below the $x$-axis.

Second, we would like our construction to preserve certain distances to lines. To be more precise, first let
\[
D^\perp(\phi_1, \phi_2)
=
\{ (r \sin\theta, r \cos\theta) : r \geq 0, \theta \in  [\phi_2 - \tfrac{\pi}{2}, \phi_1 + \tfrac{\pi}{2}] \}
.
\]
Our second desired property is that for any point $p \in D^\perp(\phi_1, \phi_2)$ and any line $\ell \subset D$, there is a line in some $\widetilde D_i$ which has the same distance to $p$ as $\ell$ does.

For a non-horizontal line $\ell\subset \R^2$ and $p \in \R^2$, we define $\dists(p, \ell)$ to be the signed distance from $p$ to $\ell$. The sign is positive if $p$ is on the left of $\ell$, and negative if $p$ is on the right. In our construction, we will always consider only lines whose direction belongs to the original cone $D(\phi_1,\phi_2)$. In particular, they are never horizontal so the signed distance is defined. The essential property of $D^\perp(\phi_1, \phi_2)$ is that for any $p \in D^\perp(\phi_1, \phi_2)$, the map $\ell \mapsto \dists(p, \ell)$ is an increasing function as $\ell$ rotates from one boundary line $\ell_1$ of $D$ to the other boundary line $\ell_2$. Hence,
\[
\{\dists(p, \ell) : \ell \subset D\}
=
[\dists(p, \ell_1), \dists(p, \ell_2)].
\]

Before stating the third and final property, we observe that since  $v_i \in D_i \cap \{y_2 < 0\}$ for all $i$, we have $D \cap \{y_2 \geq 0\} \subset (\bigcup_{i} \widetilde D_i) \cap \{y_2 \geq 0\}$. The third desired property is that the reverse containment holds if we thicken $D$ slightly. That is, $(\bigcup_{i} \widetilde D_i) \cap \{y_2 \geq 0\}$ should be contained in a small neighborhood of $D \cap \{y_2 \geq 0\}$.

The following lemma asserts that it is indeed possible to partition $D$ and translate the pieces to achieve the three desired properties above. %

\begin{lemma}%
\label{lemma:iterated-partition}
Let $-\frac{\pi}{2} < \phi_1 < \phi_2 < \frac{\pi}{2}$,  $D = D(\phi_1, \phi_2)$, $R > 0$, and $\eps > 0$. Then we can choose the partition $D = \bigcup D_i$ and the translates $\widetilde D_i = v_i + D_i$ so that
\begin{enumerate}
\item\label{condition-small-iter}
$|(\bigcup_i \widetilde D_i) \cap \{ -R \leq y_2 \leq 0 \}| < \eps$.
\item\label{condition-distance-iter}
If $p \in D^\perp$ and $\ell_0 \subset D$ is a line, then there is a line $\widetilde \ell$ in some $\widetilde D_i$ such that $\dists(p, \widetilde \ell) = \dists(p, \ell_0)$.
\item\label{condition-nbhd-iter}
$(\bigcup_i \widetilde D_i) \cap \{y_2 \geq 0\}$ is contained in the $\eps$-neighborhood of $D \cap \{ y_2 \geq 0\}$.
\end{enumerate}
\end{lemma}

To prove this lemma, we  first need a more elementary construction, in which we translate each $D_i$ downwards by only a small amount $\delta$.

\begin{lemma}%
\label{lemma:basic-partition}
Let $-\frac{\pi}{2} < \phi_1 < \phi_2 < \frac{\pi}{2}$, $D = D(\phi_1, \phi_2)$, $\delta > 0$, and $\eps > 0$. Then we can choose the partition $D = \bigcup D_i$ and the translates $\widetilde D_i = v_i + D_i$ so that
\begin{enumerate}
\item\label{condition-small-basic}
$|(\bigcup_i \widetilde D_i) \cap \{ -\delta \leq y_2 \leq 0 \}| \leq c \delta^2$, where $c = |D \cap \{0 \leq y_2 \leq 1\}|$.
\item\label{condition-distance-basic}
If $p \in D^\perp$ and $\ell_0 \subset D$ is a line, then there is a line $\widetilde \ell$ in some $\widetilde D_i$ such that $\dists(p, \widetilde \ell) = \dists(p, \ell_0)$.
\item\label{condition-vertex-basic}
For each $i$, $v_i \in \{ y_2 = -\delta \}$.
\item\label{condition-contain-basic}
For each $i$, $D^\perp \subset \widetilde D_i^\perp$.
\item\label{condition-nbhd-basic}
$(\bigcup_i \widetilde D_i) \cap \{y_2 \geq 0\}$ is contained in the $\eps$-neighborhood of $D \cap \{ y_2 \geq 0\}$.
\end{enumerate}
(If $D_i = D(\psi_1, \psi_2)$, then $D_i^\perp := D^\perp(\psi_1, \psi_2)$ and $\widetilde D_i^\perp := v_i + D_i^\perp$.)
\end{lemma}

\begin{proof}
We claim that for any partition $D = \bigcup_i D_i$, if we choose any $v_i \in \{y_2 = -\delta\} \cap D_i \cap (-D_i^\perp)$, then we have \eqref{condition-small-basic}, \eqref{condition-distance-basic}, \eqref{condition-vertex-basic}, and \eqref{condition-contain-basic}. Indeed, \eqref{condition-vertex-basic} is immediate. Since $-v_i \in D_i^\perp$, we have $D^\perp \subset D_i^\perp \subset \widetilde D_i^\perp$, so \eqref{condition-contain-basic} holds. And \eqref{condition-vertex-basic} implies \eqref{condition-small-basic} since 
$
|(\bigcup_i \widetilde D_i) \cap \{-\delta \leq y_2 \leq 0\}|
\leq
\sum_i |\widetilde D_i \cap \{-\delta \leq y_2 \leq 0\}|
=
\sum_i |D_i \cap \{0\leq y_2 \leq \delta\}|
=
c\delta^2.$

To show \eqref{condition-distance-basic} holds, let $p \in D^\perp$ and $\ell_0 \subset D$. Then $\ell_0$ is in some $D_i$. Let $\ell_1, \ell_2$ be the two boundary lines of $D_i$ with $\dists(p, \ell_1) < \dists(p, \ell_2)$. Recall that $\ell \subset D_i$ if and only if $v_i + \ell \subset \widetilde D_i$. Since $p \in D_i^\perp$ and $p \in \widetilde D_i^\perp$, we have 
$\{\dists(p, \ell) : \ell \subset D_i\}
=
[\dists(p, \ell_1), \dists(p, \ell_2)]$
and
$\{\dists(p, \ell) : \ell \subset \widetilde D_i\}
=
[\dists(p, v_i+\ell_1), \dists(p, v_i+\ell_2)]$.
Since $-v_i \in D_i \cap \{y_2 \geq 0\}$, we have
$[\dists(p, \ell_1), \dists(p, \ell_2)] \subset [\dists(p, v_i+\ell_1), \dists(p, v_i + \ell_2)]$. Thus,
\[
\dists(p, \ell_0) 
\in
\{\dists(p, \ell) : \ell \subset D_i\}
\subset
\{\dists(p, \ell) : \ell \subset \widetilde D_i\}
,
\]
so there is some $\widetilde \ell \subset \widetilde D_i$ such that $\dists(p, \widetilde \ell) = \dists(p, \ell_0)$, which completes the proof of \eqref{condition-distance-basic} and hence our claim. Finally, by making the partition $\bigcup_i D_i$ sufficiently fine and choosing $v_i$ as above, we can ensure that \eqref{condition-nbhd-basic} holds.
\end{proof}

\begin{proof}[Proof of Lemma~\ref{lemma:iterated-partition}]
We fix a large $N$ and repeatedly apply Lemma~\ref{lemma:basic-partition} with $\delta = R/N$ until the vertex of each double cone lies in $\{ y_2 = -R\}$. That is, we apply Lemma~\ref{lemma:basic-partition} once on $D$ to get $E_1$, a union of double cones with vertices in $\{ y_2 = -\delta\}$ and such that $|E_1 \cap \{-\delta \leq y_2 \leq 0 \}| < c' \delta^2$, where $c' = 2|D \cap \{0 \leq y_2 \leq 1 \}|$. Next, we apply Lemma~\ref{lemma:basic-partition} to every double cone in $E_1$ to get $E_2$, a union of double cones with vertices in $\{ y_2 = -2\delta\}$ and such that $|E_2 \cap \{-2\delta \leq y_2 \leq -\delta \}| < c' \delta^2$. By Lemma~\ref{lemma:basic-partition}\eqref{condition-nbhd-basic}, we can also ensure that $|E_2 \cap \{-\delta \leq y_2 \leq 0 \}| < c' \delta^2$.  

We continue in this way to obtain $E_1, \ldots, E_N$, such that $|E_k \cap \{-j\delta \leq y_2 \leq -(j-1)\delta \}| < c' \delta^2$ for $1 \leq j \leq k \leq N$. Because of Lemma~\ref{lemma:basic-partition}\eqref{condition-nbhd-basic}, we can also ensure that $E_k \cap \{y_2 \geq 0\}$ is in the $\eps$-neighborhood of $D \cap \{y_2 \geq 0\}$ for each $k$. Ultimately, we have $|E_N \cap \{-R \leq y_2 \leq 0 \}| \leq N c' \delta^2 = c'R^2/N$. By choosing $N$ sufficiently large, we can make this quantity as small as we wish. Writing $E_N$ as $\bigcup_i \widetilde D_i$, we obtain \eqref{condition-small-iter} and \eqref{condition-nbhd-iter}. Furthermore, every time we translate downwards by $\delta$, Lemma~\ref{lemma:basic-partition}\eqref{condition-contain-basic} allows us to apply Lemma~\ref{lemma:basic-partition}\eqref{condition-distance-basic}. Thus, \eqref{condition-distance-iter} holds.
\end{proof}

\subsection{Proof of Theorem~\ref{theorem:typical-davies-nikodym}}

First we apply our main geometric construction from the previous section to prove the following lemma.

\begin{lemma}
\label{lemma:nbhd-SOn}
Let $n\ge 2$, $(x_0,r_0,T_0) \in \R^n \times [0, \infty) \times SO(n)$, let $B \subset \R^n$ be a closed ball, and suppose that $x_0 + r_0T_0(e_1) \not\in B$. 
Let $\eta > 0$ be arbitrary. Then there is a (relatively)  open neighborhood $\widetilde U$ of $(x_0,r_0)$ in $\R^n \times [0, \infty)$ such that for each $\eps > 0$, there is a set $\widetilde D \subset \R^n$ such that:
\begin{enumerate}
\item For all $(x,r) \in \widetilde U$, there is an affine hyperplane $V \subset \widetilde D$ of distance $r$ from $x$ and such that the angle between $V$ and $T_0(H)$ is at most $\eta$.
\item $|B \cap \widetilde D| < \eps$.
\end{enumerate}
\end{lemma}

\begin{proof}

First we show the lemma for $n=2$. Since $x_0 + r_0T_0(e_1) \not\in B$, without loss of generality, we may  assume that $T_0 \in SO(2)$ is the identity, that $x_0+r_0e_1 \in \{y_1 = 0, y_2 > 0\}$, and that $B$ does not intersect $\{y_1 = 0, y_2 \ge 0\}$. We can also assume that $B$ lies in $\{y_2 \geq - 2\diam B\}$. %
It follows that $x_0$ lies in the upper half-plane $\{y_2 > 0\}$.

Using the notation from Section \ref{subsec:cones}, let $D = D(-\phi, \phi)$ be a double cone, 
where $\phi \in (0,\eta)$ is small enough so that $x_0 \in D^\perp$ and $B \cap D \cap \{y_2 \geq 0\} = \emptyset$.
The boundary of $D$ is made up of two lines, $\ell_1, \ell_2$, with $\dists(x_0, \ell_1) < r_0 < \dists(x_0, \ell_2)$. Let $\rho > 0$ be sufficiently small so that $\dists(x_0, \ell_1) < r_0 - \rho$ and $r_0 + \rho < \dists(x_0, \ell_2)$. 
Let $\widetilde U$ be a (relatively) open neighborhood of $(x_0, r_0)$ contained in
\[
\{ y \in D^\perp : \dists(y, \ell_1) < r_0 - \rho \text{ and } r_0 + \rho < \dists(y, \ell_2)\} \times (r_0-\rho, r_0+\rho)
.\]
Then for any $(x, r) \in \widetilde U$, there is a line $\ell \subset D$ of signed distance $r$ from $x$. Given $\eps > 0$, we apply Lemma~\ref{lemma:iterated-partition} to get $\widetilde D := \bigcup_i \widetilde D_i$ with $|\widetilde D \cap \{ -2\diam B \leq y_2 \leq 0 \}| < \eps$ and $B \cap \widetilde D \cap \{y_2 \geq 0\} = \emptyset$. It follows that $|B \cap \widetilde D| < \eps$. By Lemma~\ref{lemma:iterated-partition}\eqref{condition-distance-iter}, 
for every $(x, r) \in \widetilde U$, there is some line $\ell \subset \widetilde D$ of distance $r$ from $x$. Every line $\ell \subset \widetilde D$ is a translate of some line in $D$, so the angle between $\ell$ and $H$ is at most $\phi < \eta$. This completes the proof in dimension $n=2$.

For an arbitrary $n\ge 2$, we can assume without loss of generality that $T_0$ is the identity, that $x_0$ (and hence also $x_0+r_0 e_1$) is contained in the two-dimensional plane $\R^2\subset\R^n$ defined by the first two coordinate axes, and that the same assumptions hold as in the first paragraph of our proof. Then, if we project the ball $B$ into $\R^2$, take the sets $\widetilde U,\widetilde D\subset \R^2$ constructed above, and multiply them by $\R^{n-2}$, the resulting sets satisfy the requirements of the statement of Lemma~\ref{lemma:nbhd-SOn} with $\eps$ replaced by $\eps\diam(B)^{n-2}$.
\end{proof}

\begin{lemma}
\label{lemma:nbhd-of-x-r}
Let $(x_0,r_0,T_0) \in \R^n \times [0, \infty) \times SO(n)$, and let $B \subset \R^n$ be a closed ball. %
Let $G$ be a (relatively) open neighborhood of $(x_0,r_0,T_0)$ in $\R^n \times [0, \infty) \times SO(n)$. Then there is an open neighborhood $U \subset \R^n \times [0, \infty)$ of $(x_0,r_0)$ such that for each $\eps > 0$, there is a compact set $K \subset G$ with full projection onto $U$ and such that 
\begin{equation}
\label{eq:small-intersection-B}
B \cap \bigcup_{\substack{(x, r, T) \in K\\ x+ rT(e_1) \not\in 2B}}(x+rT(e_1) + T(H))
\end{equation}
has measure less than $\eps$. (Here $2B$ denotes the closed ball with the same center as $B$ and with twice the radius.)
\end{lemma}

\begin{proof}
Without loss of generality, we may assume $G = G_1 \times G_2$, where $G_1$ and $G_2$ are open sets in $\R^n \times [0, \infty)$ and $SO(n)$, respectively.

If $x_0 + r_0T_0(e_1) \in B$, then we can choose $K \subset G$ to contain a neighborhood of $(x_0,r_0,T_0)$ and such that $x+rT(e_1) \in 2B$ for all $(x, r, T) \in K$. Then the set \eqref{eq:small-intersection-B} is empty, so the lemma holds trivially.

Now suppose $x_0 + r_0T_0(e_1) \not\in B$. We can apply the previous lemma with $\eta$ sufficiently small (depending on $G_2$) to get a set $\widetilde U$. We take $U$ to be an open neighborhood of $(x_0, r_0)$ inside $\widetilde U$ and compactly contained in $G_1$. Then for each $\eps > 0$, the previous lemma gives a set $\widetilde D$. We take $K$ to be the closure of
\[
\{(x, r, T) \in U \times SO(n) :  x + rT(e_1) + T(H) \subset \widetilde D
\},
\]
and by the properties of $\widetilde D$ given by the previous lemma, this $K$ has the desired properties.
\end{proof}

 For $B \subset \R^n$ a closed ball, let $\cA(B)$ be the set of all $K \in \cK''$ such that 
\[
B \cap \bigcup_{\substack{(x, r, T) \in K\\ x + rT(e_1) \not\in 2B}}(x+rT(e_1) + T(H))
\]
has measure zero.

\begin{lemma}
\label{lemma:2-diam-Q-dense}
$\cA(B)$ is residual in $\cK''$.
\end{lemma}

\begin{proof}
For $\eps > 0$, let $\cA(B, \eps)$ be the set of those $K \in \cK''$ for which there is an $\eta > 0$ such that
\[
B \cap \bigcup_{\substack{(x, r, T) \in K^\eta\\ x + rT(e_1) \not\in 2B}}(x+rT(e_1) + T(H))
\]
has measure less than $\eps$, where $K^\eta$ denotes the open $\eta$-neighborhood of $K$. Since $\cA(B) = \bigcap_{m=1}^\infty \cA(B, \frac{1}{m})$, it is enough to show that $\cA(B, \eps)$ is open and dense in $\cK''$ for each $\eps > 0$. 

Fix $\eps > 0$. $\cA(B, \eps)$  is clearly open in $\cK''$. To show that it is dense, let $L \in \cK''$ be arbitrary. Our aim is to find a $K \in \cA(B, \eps)$ arbitrarily close to $L$. For each $(x, r, T) \in L$, we take a neighborhood $G_{(x, r, T)}$ of $(x,r,T)$, which we choose sufficiently small (to be specified later). Then we apply Lemma~\ref{lemma:nbhd-of-x-r} to $(x, r, T)$ to get a neighborhood $U_{(x, r, T)} \subset \R^n \times [0, \infty)$ of $(x, r)$. By compactness, there is a finite collection $\{(x_i, r_i, T_i)\} \subset L$ such that $\{U_{(x_i, r_i, T_i)}\}$ covers $\J$.

Choose $\eps_i$ so that $\sum_i \eps_i < \eps$. We apply Lemma~\ref{lemma:nbhd-of-x-r} to each $U_{(x_i, r_i, T_i)}$ with $\eps_i$ in place of $\eps$ to get a compact $K_i \subset G_{(x_i, r_i, T_i)}$ with full projection onto $U_{(x_i, r_i, T_i)}$. Let $\widetilde K_i = K_i \cap (\J \times SO(n))$. Let $K$ be the union of $\bigcup_i \widetilde K_i$ together with a finite $\delta$-net of $L$. Then $K \in \cA(B, \eps)$. By choosing $\delta$ and all the $G_{(x, r, T)}$ sufficiently small, we can make $K$ and $L$ arbitrarily close to each other in the Hausdorff metric.
\end{proof}

Now we are ready to prove Theorem~\ref{theorem:typical-davies-nikodym}. It follows easily from Lemma~\ref{lemma:2-diam-Q-dense} that, for a typical $K \in \cK''$,
\begin{equation}
\label{eq:union-punctured}
\bigcup_{(x, r, T) \in K}(x+rT(e_1) + T(H\setminus\{0\}))
\end{equation}
has measure zero. Indeed, let $\{B_i\}$ be a countable collection of balls such that every point in $\R^n$ is covered by a ball of arbitrarily small diameter, and suppose that $K\in\bigcap_i\mathcal A(B_i)$. 
For every $(x,r,T)\in K$ and for every $y\in H\setminus\{0\}$ there is a $B_i$ which contains $x + rT(e_1) + T(y)$ and has diameter less than $|y|/2$. Then $x + rT(e_1) \not\in 2B_i$, so $x + rT(e_1) + T(y)$ belongs to the null set $$B_i \cap \bigcup_{\substack{(x, r, T) \in K\\ x + rT(e_1) \not\in 2B_i}}(x+rT(e_1) + T(H)).$$

To complete the proof of Theorem~\ref{theorem:typical-davies-nikodym}, we need to show that we can remove the puncture from $H \setminus \{0\}$ when the distance $r$ is nonzero. By adapting the argument in the proof of Lemma~\ref{l:rotatedeverysize}, we can show that for any $r_0 > 0$, for a typical $K \in \cK''$, the set
\begin{equation}
\label{eq:union-gtr-r0}
\bigcup_{\substack{(x, r, T) \in K\\ r\geq r_0}}x+rT(e_1)
\end{equation} has dimension at most $1$, hence measure zero. By taking a countable intersection of $r_0$ tending to $0$, we see that for a typical $K \in \cK''$, the set
$$\bigcup_{\substack{(x, r, T) \in K\\ r\neq 0}}x+rT(e_1)$$ has measure zero. This completes the proof of Theorem~\ref{theorem:typical-davies-nikodym}.

\begin{remark}
The argument in the proof of Lemma~\ref{l:rotatedeverysize} cannot be applied directly to show that the set \eqref{eq:union-gtr-r0} has dimension at most $1$ for a typical $K \in \cK''$. There is a slight complication due to the fact that for $r_0 > 0$, the function $\cK'' \to \cK^n$ defined by $K \mapsto \bigcup_{(x, r, T) \in K, r \geq r_0} x + rT(e_1)$ is not necessarily continuous. However, the technical modifications required are straightforward, so we leave this to the reader. 
\end{remark}

\end{document}